\documentclass[a4paper]{amsart}
\usepackage{color}
\usepackage[latin1]{inputenc}
\usepackage[T1]{fontenc}
\usepackage{amsfonts}
\usepackage{amssymb}
\usepackage{amsmath}
\usepackage{amsthm}
\usepackage{stmaryrd}
\usepackage{eufrak}
\usepackage{graphicx,type1cm,eso-pic,color}
\usepackage{pstricks,pst-plot,pstricks-add}
\usepackage{float}
\newtheorem{theorem}{Theorem}[section]
\newtheorem{lemma}[theorem]{Lemma}
\newtheorem{proposition}[theorem]{Proposition}
\newtheorem{corollary}[theorem]{Corollary}
\newtheorem{definition}[theorem]{Definition}
\newtheorem{assumption}[theorem]{Assumption}

\begin{document}
\setlength\arraycolsep{2pt}
\title{Nonparametric Estimation in Fractional SDE}
\author{Fabienne COMTE*}
\address{*Laboratoire MAP5, Universit\'e Paris Descartes, Paris, France}
\email{fabienne.comte@parisdescartes.fr}
\author{Nicolas MARIE$^\dag$}
\address{$^\dag$Laboratoire Modal'X, Universit\'e Paris Nanterre, Nanterre, France}
\email{nmarie@parisnanterre.fr}
\address{$^\dag$ESME Sudria, Paris, France}
\email{nicolas.marie@esme.fr}
\keywords{}
\date{}
\maketitle
\noindent
%


%
\begin{abstract}
This paper deals with the consistency and a rate of convergence for a Nadaraya-Watson estimator of the drift function of a stochastic differential equation driven by an additive fractional noise. The results of this paper are obtained via both some long-time behavior properties of Hairer and some properties of the Skorokhod integral with respect to the fractional Brownian motion. These results are illustrated on the fractional Ornstein-Uhlenbeck process.
\end{abstract}
%
\tableofcontents
%


%
\section{Introduction}
Consider the stochastic differential equation
\begin{equation}\label{main_equation}
X(t) =
X_0 +\int_{0}^{t}b(X(s))ds +
\sigma B(t),
\end{equation}
where $B$ is a fractional Brownian motion of Hurst index $H\in ]1/2,1[$, $b :\mathbb R\rightarrow\mathbb R$ is a continuous map and $\sigma\in\mathbb R^*$.
\\
\\
Along the last two decades, many authors studied statistical inference from observations drawn from stochastic differential equations driven by fractional Brownian motion.
\\
Most references on the estimation of the trend component in Equation (\ref{main_equation}) deals with parametric estimators. In Kleptsyna and Le Breton \cite{KB01} and Hu and Nualart \cite{HN10}, estimators of the trend component in Langevin's equation are studied. Kleptsyna and Le Breton \cite{KB01} provide a maximum likelihood estimator, where the stochastic integral with respect to the solution of Equation (\ref{main_equation}) returns to an It\^o integral. In \cite{TV07}, Tudor and Viens extend this estimator to equations with a drift function depending linearly on the unknown parameter. Hu and Nualart \cite{HN10} provide a least square estimator, where the stochastic integral with respect to the solution of Equation (\ref{main_equation}) is taken in the sense of Skorokhod. In \cite{HNZ18}, Hu, Nualart and Zhou extend this estimator to equations with a drift function depending linearly on the unknown parameter.
\\
In Tindel and Neuenkirch \cite{NT14}, the authors study a least square-type estimator defined by an objective function tailor-maid with respect to the main result of Tudor and Viens \cite{TV09} on the rate of convergence of the quadratic variation of the fractional Brownian motion. In \cite{CT13}, Chronopoulou and Tindel provide a likelihood based numerical procedure to estimate a parameter involved in both the drift and the volatility functions in a stochastic differential equation with multiplicative fractional noise.
\\
On the nonparametric estimation of the trend component in Equation (\ref{main_equation}), there are only few references. Saussereau \cite{SAUSSEREAU14} and Mishra and Prakasa Rao \cite{MP11} study the consistency of some Nadaraya-Watson's-type estimators of the drift function $b$ in Equation (\ref{main_equation}). On the nonparametric estimation in It\^o's calculus framework, the reader is referred to Kutoyants \cite{KUTOYANTS04}.
\\
\\
Let $K :\mathbb R\rightarrow\mathbb R_+$ be a kernel that is a nonnegative function with integral equal to $1$. The paper deals with the consistency and a rate of convergence for the Nadaraya-Watson estimator
\begin{equation}\label{main_estimator}
\widehat b_{T,h}(x) :=
\frac{\displaystyle{\int_{0}^{T}K\left(\frac{X(s) - x}{h}\right)\delta X(s)}}
{\displaystyle{\int_{0}^{T}K\left(\frac{X(s) - x}{h}\right)ds}}
\textrm{ $;$ }
x\in\mathbb R
\end{equation}
of the drift function $b$ in Equation (\ref{main_equation}), where the stochastic integral with respect to $X$ is taken in the sense of Skorokhod. Since to compute the Skorokhod integral is a challenge, by denoting by $X_{x_0}$ the solution of Equation (\ref{main_equation}) with initial condition $x_0\in\mathbb R$, the following estimator is also studied:
\begin{small}
\begin{eqnarray}\nonumber 
 \widehat b_{T,h,\varepsilon}(x) & := &
 \frac{\displaystyle{\int_{0}^{T}K\left(\frac{X_{x_0}(s) - x}{h}\right)d X_{x_0}(s)}}
 {\displaystyle{\int_{0}^{T}K\left(\frac{X_{x_0}(s) - x}{h}\right)ds}}\\
 \label{b_epsilon}
 & &
 -\alpha_H\sigma^2\frac{
 \displaystyle{\frac{1}{h}\int_{0}^{T}
 \int_{0}^{u}K'\left(\frac{X_{x_0}(u) - x}{h}\right)\frac{X_{x_0 +\varepsilon}(u) - X_{x_0}(u)}{X_{x_0 +\varepsilon}(v) - X_{x_0}(v)}|u - v|^{2H - 2}dvdu}}{\displaystyle{\int_{0}^{T}K\left(\frac{X_{x_0}(s) - x}{h}\right)ds}}
\end{eqnarray}
\end{small}
\newline
with $\varepsilon > 0$ and $x\in\mathbb R$. In this second estimator, the stochastic integral is taken pathwise. It depends on $H$, but an estimator of this parameter is for instance provided in Kubilius and Skorniakov \cite{KS16}.
\\
As detailed in Subsection 2.2, the Skorokhod integral is defined via the divergence operator which is the adjoint of the Malliavin derivative for the fractional Brownian motion. If $H = 1/2$, then the Skorokhod integral coincides with It\^o's integral on its domain. When $H\in ]1/2,1[$, it is more difficult to compute the Skorokhod integral, but not impossible as explained at the end of Subsection 2.2. Note that, the pathwise stochastic integral defined in Subsection 2.1 would have been a more natural choice, but unfortunately, it does not provide a consistent estimator (see Proposition \ref{pathwise_NW}).
\\
 Clearly, to be computable, the estimator $\widehat b_{T,h,\varepsilon}(x)$ requires an observed path of the solution of Equation (\ref{main_equation}) for two close but different values of the initial condition. This is not possible in any context, but we have in mind the following application field: if $t\mapsto X_{x_0}(\omega,t)$ denotes the concentration of a drug along time during its elimination by a patient $\omega$ with initial dose $x_0 > 0$, $t\mapsto X_{x_0 +\varepsilon}(\omega,t)$ could be approximated by replicating the exact same protocol on patient $\omega$, but with initial dose $x_0 +\varepsilon$ after the complete elimination of the previous dose.
\\
We mention that we do not study the additional error which occurs when only discrete time observations with step $\Delta$ on $[0, T]$ ($T = n\Delta$) are available. Formula (\ref{b_epsilon}) has then to be discretized and a study in the spirit of Saussereau \cite{SAUSSEREAU14} (Section 4.3) must be conducted.
\\
Section 2 deals with some preliminary results on stochastic integrals with respect to the fractional Brownian motion and an ergodic theorem for the solution of Equation (\ref{main_equation}). The consistency and a rate of convergence of the Nadaraya-Watson estimator studied in this paper are stated in Section 3. Almost all the proofs of the paper are provided in Section 4.
\\
\\
\textbf{Notations:}
\begin{enumerate}
 \item The vector space of Lipschitz continuous maps from $\mathbb R$ into itself is denoted by $\textrm{Lip}(\mathbb R)$ and equipped with the Lipschitz semi-norm $\|.\|_{\textrm{Lip}}$ defined by
 \begin{displaymath}
 \|\varphi\|_{\textrm{Lip}} :=
 \sup\left\{
 \frac{|\varphi(y) -\varphi(x)|}{|y - x|}
 \textrm{ ; }
 x,y\in\mathbb R
 \textrm{ and }
 x\not= y\right\}
 \end{displaymath}
 for every $\varphi\in\textrm{Lip}(\mathbb R)$.
 \item For every $m\in\mathbb N$,
 \begin{displaymath}
 C_{b}^{m}(\mathbb R) :=
 \left\{\varphi\in C^m(\mathbb R) :
 \max_{k\in\llbracket 0,m\rrbracket}
 \|\varphi^{(k)}\|_{\infty} <\infty\right\}.
 \end{displaymath}
 \item For every $m\in\mathbb N^*$,
 \begin{displaymath}
 \textrm{Lip}_{b}^{m}(\mathbb R) :=
 \left\{\varphi\in C^m(\mathbb R) :
 \varphi\in\textrm{Lip}(\mathbb R)
 \textrm{ and }
 \max_{k\in\llbracket 1,m\rrbracket}
 \|\varphi^{(k)}\|_{\infty} <\infty\right\}
 \end{displaymath}
 and for every $\varphi\in\textrm{Lip}_{b}^{m}(\mathbb R)$,
 \begin{displaymath}
 \|\varphi\|_{\textrm{Lip}_{b}^{m}} :=
 \|\varphi\|_{\textrm{Lip}}\vee
 \max_{k\in\llbracket 1,m\rrbracket}\|\varphi^{(k)}\|_{\infty}.
 \end{displaymath}
 The map $\|.\|_{\textrm{Lip}_{b}^{m}}$ is a semi-norm on $\textrm{Lip}_{b}^{m}(\mathbb R)$.
 \\
 \\
 Note that for every $m\in\mathbb N^*$,
 \begin{displaymath}
 C_{b}^{m}(\mathbb R)
 \subset\textrm{Lip}_{b}^{m}(\mathbb R).
 \end{displaymath}
 \item Consider $n\in\mathbb N^*$. The vector space of infinitely continuously differentiable maps $f :\mathbb R^n\rightarrow\mathbb R$ such that $f$ and all its partial derivatives have polynomial growth is denoted by $C_{p}^{\infty}(\mathbb R^n,\mathbb R)$.
\end{enumerate}
%


%
\section{Stochastic integrals with respect to the fractional Brownian motion and an ergodic theorem for fractional SDE}
On the one hand, this section presents two different methods to define a stochastic integral with respect to the fractional Brownian motion. The first one is based on the pathwise properties of the fractional Brownian motion. Even if this approach is very natural, it is proved in Proposition \ref{pathwise_NW} that the pathwise stochastic integral is not appropriate to get a consistent estimator of the drift function $b$ in Equation (\ref{main_equation}). Another stochastic integral with respect to the fractional Brownian motion is defined via the Malliavin divergence operator. This stochastic integral is called Skorokhod's integral with respect to $B$. If $H = 1/2$, which means that $B$ is a Brownian motion, the Skorokhod integral defined via the divergence operator coincides with It\^o's integral on its domain. This integral is appropriate for the estimation of the drift function $b$ in Equation (\ref{main_equation}).  On the other hand, an ergodic theorem for the solution of Equation (\ref{main_equation}) is stated in Subsection 2.3.
%


%
\subsection{The pathwise stochastic integral}
This subsection deals with some definitions and basic properties of the pathwise stochastic integral with respect to the fractional Brownian motion of Hurst index greater than $1/2$.
%


%
\begin{definition}\label{Riemann_sum}
Consider $x$ and $w$ two continuous functions from $[0,T]$ into $\mathbb R$. Consider a partition $D := (t_k)_{k\in\llbracket 0,m\rrbracket}$ of $[s,t]$ with $m\in\mathbb N^*$ and $s,t\in [0,T]$ such that $s < t$. The Riemann sum of $x$ with respect to $w$ on $[s,t]$ for the partition $D$ is
\begin{displaymath}
J_{x,w,D}(s,t) :=
\sum_{k = 0}^{m - 1}x(t_k)(w(t_{k + 1}) - w(t_k)).
\end{displaymath}
\end{definition}
\noindent
\textbf{Notation.} With the notations of Definition \ref{Riemann_sum}, the mesh of the partition $D$ is
\begin{displaymath}
\delta(D) :=
\max_{k\in\llbracket 0,m - 1\rrbracket}
|t_{k + 1} - t_k|.
\end{displaymath}
The following theorem ensures the existence and the uniqueness of Young's integral (see Friz and Victoir \cite{FV10}, Theorem 6.8).
%


%
\begin{theorem}\label{Young_integral}
Let $x$ (resp. $w$) be a $\alpha$-H\"older (resp. $\beta$-H\"older) continuous map from $[0,T]$ into $\mathbb R$ with $\alpha,\beta\in ]0,1]$ such that $\alpha +\beta > 1$. There exists a unique continuous map $J_{x,w} : [0,T]\rightarrow\mathbb R$ such that for every $s,t\in [0,T]$ satisfying $s < t$ and any sequence $(D_n)_{n\in\mathbb N}$ of partitions of $[s,t]$ such that $\delta(D_n)\rightarrow 0$ as $n\rightarrow\infty$,
\begin{displaymath}
\lim_{n\rightarrow\infty}
|J_{x,w}(t) - J_{x,w}(s) - J_{x,w,D_n}(s,t)| = 0.
\end{displaymath}
The map $J_{x,w}$ is the Young integral of $x$ with respect to $w$ and $J_{x,w}(t) - J_{x,w}(s)$ is denoted by
\begin{displaymath}
\int_{s}^{t}x(u)dw(u)
\end{displaymath}
for every $s,t\in [0,T]$ such that $s < t$.
\end{theorem}
\noindent
The following proposition is a change of variable for Young's integral.
%


%
\begin{proposition}\label{change_of_variable}
Let $x$ be a $\alpha$-H\"older continuous map from $[0,T]$ into $\mathbb R$ with $\alpha\in ]1/2,1[$. For every $\varphi\in\normalfont{\textrm{Lip}}_{b}^{1}(\mathbb R)$ and $s,t\in [0,T]$ such that $s < t$,
\begin{displaymath}
\varphi(x(t)) -\varphi(x(s)) =
\int_{s}^{t}\varphi'(x(u))dx(u).
\end{displaymath}
\end{proposition}
\noindent
For any $\alpha\in ]1/2,H[$, the paths of $B$ are $\alpha$-H\"older continuous (see Nualart \cite{NUALART06}, Section 5.1). So, for every process $Y := (Y(t))_{t\in [0,T]}$ with $\beta$-H\"older continuous paths from $[0,T]$ into $\mathbb R$ such that $\alpha +\beta > 1$, by Theorem \ref{Young_integral}, it is natural to define the pathwise stochastic integral of $Y$ with respect to $B$ by
\begin{displaymath}
\left(\int_{0}^{t}Y(s)dB(s)\right)(\omega) :=
\int_{0}^{t}Y(\omega,s)dB(\omega,s)
\end{displaymath}
for every $\omega\in\Omega$ and $t\in [0,T]$.
%


%
\subsection{The Skorokhod integral}
This subsection deals with some definitions and results on Malliavin calculus in order to define and to provide a suitable expression of Skorokhod's integral.
\\
\\
Consider the vector space
\begin{displaymath}
\mathcal H :=
\left\{\varphi :\mathbb R_+\rightarrow\mathbb R :
\int_{0}^{\infty}
\int_{0}^{\infty}
|t - s|^{2H - 2}|\varphi(s)|\cdot|\varphi(t)|dsdt <\infty\right\}.
\end{displaymath}
Equipped with the scalar product
\begin{displaymath}
\langle\varphi,\psi\rangle_{\mathcal H} :=
H(2H - 1)\int_{0}^{\infty}\int_{0}^{\infty}
|t - s|^{2H - 2}\varphi(s)\psi(t)dsdt
\textrm{ ; }
\varphi,\psi\in\mathcal H,
\end{displaymath}
$\mathcal H$ is the reproducing kernel Hilbert space of $B$. Let $\mathbf B$ be the map defined on $\mathcal H$ by
\begin{displaymath}
\mathbf B(h) :=
\int_{0}^{.}h(s)dB(s)
\textrm{ $;$ } h\in\mathcal H
\end{displaymath}
which is the Wiener integral of $h$ with respect to $B$. The family $(\mathbf B(h))_{h\in\mathcal H}$ is an isonormal Gaussian process.
%


%
\begin{definition}\label{Malliavin_derivative}
The Malliavin derivative of a smooth functional
\begin{displaymath}
F = f(
\mathbf B(h_1),\dots,
\mathbf B(h_n))
\end{displaymath}
where $n\in\mathbb N^*$, $f\in C_{p}^{\infty}(\mathbb R^n,\mathbb R)$ and $h_1,\dots,h_n\in\mathcal H$ is the $\mathcal H$-valued random variable
\begin{displaymath}
\mathbf DF :=
\sum_{k = 1}^{n}
\partial_k f
(
\mathbf B(h_1),\dots,
\mathbf B(h_n))h_k.
\end{displaymath}
\end{definition}
%


%
\begin{proposition}\label{Malliavin_derivative_domain}
The map $\mathbf D$ is closable from $L^2(\Omega,\mathcal A,\mathbb P)$ into $L^2(\Omega;\mathcal H)$. Its domain in $L^2(\Omega,\mathcal A,\mathbb P)$ is denoted by $\mathbb D^{1,2}$ and is the closure of the smooth functionals space for the norm $\|.\|_{1,2}$ defined by
\begin{displaymath}
\|F\|_{1,2}^{2} :=
\mathbb E(|F|^2) +
\mathbb E(\|\mathbf DF\|_{\mathcal H}^{2}) < \infty
\end{displaymath}
for every $F\in L^2(\Omega,\mathcal A,\mathbb P)$.
\end{proposition}
\noindent
For a proof, see Nualart \cite{NUALART06}, Proposition 1.2.1.
%


%
\begin{definition}\label{divergence_operator}
The adjoint $\delta$ of the Malliavin derivative $\mathbf D$ is the divergence operator. The domain of $\delta$ is denoted by $\normalfont{\textrm{dom}}(\delta)$ and $u\in\normalfont{\textrm{dom}}(\delta)$ if and only if there exists a deterministic constant $c > 0$ such that for every $F\in\mathbb D^{1,2}$,
\begin{displaymath}
|\mathbb E(\langle\mathbf DF,u\rangle_{\mathcal H})|
\leqslant
c\mathbb E(|F|^2)^{1/2}.
\end{displaymath}
\end{definition}
\noindent
For every process $Y := (Y(s))_{s\in\mathbb R_+}$ and every $t > 0$, if $Y\mathbf 1_{[0,t]}\in\textrm{dom}(\delta)$, then its Skorokhod integral with respect to $B$ is defined on $[0,t]$ by
\begin{displaymath}
\int_{0}^{t}Y(s)\delta B(s) :=
\delta(Y\mathbf 1_{[0,t]}).
\end{displaymath}

With the same notations:
\begin{displaymath}
\int_{0}^{t}
Y(s)\delta X(s) :=
\int_{0}^{t}Y(s)b(X(s))ds +
\sigma\int_{0}^{t}Y(s)\delta B(s).
\end{displaymath}
The following proposition provides the link between the Skorokhod integral and the pathwise stochastic integral of Subsection 2.1.
%


%
\begin{proposition}\label{expression_Skorokhod}
If $b\in\normalfont{\textrm{Lip}}_{b}^{1}(\mathbb R)$, then Equation (\ref{main_equation}) with initial condition $x\in\mathbb R$ has a unique solution $X_x$ with $\alpha$-H\"older continuous paths for every $\alpha\in ]0,H[$. Moreover, for every $\varphi\in\normalfont{\textrm{Lip}}_{b}^{1}(\mathbb R)$,
\begin{eqnarray}
 \label{decomposition_Skorokhod}
 \int_{0}^{t}\varphi(X_x(u))\delta X_x(u) & = &
 \int_{0}^{t}\varphi(X_x(u))dX_x(u)\\
 & &
 -\alpha_H\sigma^2
 \int_{0}^{t}
 \int_{0}^{u}\varphi'(X_x(u))\frac{\partial_x X_x(u)}{\partial_x X_x(v)}|u - v|^{2H - 2}dvdu,
 \nonumber
\end{eqnarray}
where $\alpha_H = H(2H - 1)$.
\end{proposition}
\noindent
Moreover, we can prove the following Corollary, which allows us to propose a computable form for the estimator.
%


%
\begin{corollary}\label{approximation_Skorokhod}
Assume that $b\in\normalfont{\textrm{Lip}}_{b}^{2}(\mathbb R)$ and there exists a constant $M > 0$ such that
\begin{displaymath}
b'(x)\leqslant -M
\textrm{ $;$ }
\forall x\in\mathbb R.
\end{displaymath}
For every $\varphi\in\normalfont{\textrm{Lip}}_{b}^{1}(\mathbb R)$, $x\in\mathbb R$ and $\varepsilon,t > 0$,
\begin{displaymath}
\left|\int_{0}^{t}\varphi(X_x(u))\delta X_x(u) -
S_{\varphi}(x,\varepsilon,t)\right|
\leqslant
C_{\varphi}\varepsilon t^{2H - 1},
\end{displaymath}
where
\begin{eqnarray*}
 S_{\varphi}(x,\varepsilon,t) & := &
 \int_{0}^{t}\varphi(X_x(u))dX_x(u)\\
 & &
 -\alpha_H\sigma^2
 \int_{0}^{t}
 \int_{0}^{u}\varphi'(X_x(u))\frac{X_{x +\varepsilon}(u) - X_x(u)}{X_{x +\varepsilon}(v) - X_x(v)}|u - v|^{2H - 2}dvdu
\end{eqnarray*}
and
\begin{displaymath}
C_{\varphi} :=
H\sigma^2\frac{\|b''\|_{\infty}\|\varphi'\|_{\infty}}{2M^2}.
\end{displaymath}
\end{corollary}
\noindent
As mentioned in the Introduction, the formula for $ S_{\varphi}(x,\varepsilon,t)$  can be used if two paths of $X$ can be observed with different but close initial conditions.\\
Lastly, the following theorem, recently proved by Hu, Nualart and Zhou in \cite{HNZ18} (see Proposition 4.4), provides a suitable control of Skorokhod's integral to study its long-time behavior.
%


%
\begin{theorem}\label{control_divergence_integral}
Assume that $b\in\normalfont{\textrm{Lip}}_{b}^{2}(\mathbb R)$ and there exists a constant $M > 0$ such that
\begin{displaymath}
b'(x)\leqslant -M
\textrm{ $;$ }
\forall x\in\mathbb R.
\end{displaymath}
There exists a deterministic constant $C > 0$, not depending on $T$, such that for every $\varphi\in\normalfont{\textrm{Lip}}_{b}^{1}(\mathbb R)$:
\begin{eqnarray*}
 \mathbb E\left(
 \left|\int_{0}^{T}\varphi(X(s))\delta B(s)\right|^2\right)
 & \leqslant &
 C\left(\left(\int_{0}^{T}\mathbb E(|\varphi(X(s))|^{1/H})ds\right)^{2H}\right.\\
 & &
 +\left.
 \left(\int_{0}^{T}\mathbb E(|\varphi'(X(s))|^2)^{1/(2H)}ds\right)^{2H}\right) <\infty.
\end{eqnarray*}
\end{theorem}
%


%
\subsection{Ergodic theorem for the solution of a fractional SDE}
On the ergodicity of fractional SDEs, the reader can refer to Hairer \cite{HAIRER05}, Hairer and Ohashi \cite{HO07} and Hairer and Pillai \cite{HP10} (see Subsection 4.3 for details).
\\
\\
In the sequel, the map $b$ fulfills the following condition.
%


%
\begin{assumption}\label{ergodicity}
The map $b$ belongs to $\normalfont{\textrm{Lip}}_{b}^{\infty}(\mathbb R)$ and there exists a constant $M > 0$ such that
\begin{equation}\label{dissipativity}
b'(x)\leqslant -M
\textrm{ $;$ }
\forall x\in\mathbb R.
\end{equation}
\end{assumption}
\noindent
\textbf{Remarks:}
\begin{enumerate}
 \item Since $b\in\textrm{Lip}_{b}^{1}(\mathbb R)$, Equation (\ref{main_equation}) has a unique solution.
 \item Under Assumption \ref{ergodicity}, the dissipativity conditions of Hairer \cite{HAIRER05}, Hairer and Ohashi \cite{HO07} and Hu, Nualart and Zhou \cite{HNZ18} are fulfilled by $b$:
 \begin{displaymath}
 (x - y)(b(x) - b(y))
 \leqslant -M(x - y)^2
 \textrm{ $;$ }
 \forall x,y\in\mathbb R
 \end{displaymath}
 and there exists a constant $M' > 0$ such that
 \begin{displaymath}
 xb(x)\leqslant
 M'(1 - x^2)
 \textrm{ $;$ }
 \forall x\in\mathbb R.
 \end{displaymath}
Therefore, Assumption \ref{ergodicity} is sufficient to apply the results proved in \cite{HAIRER05}, \cite{HO07} and \cite{HNZ18} in the sequel.
\end{enumerate}
%


%
\begin{proposition}\label{ergodic_theorem}
Consider a measurable map $\varphi :\mathbb R\rightarrow\mathbb R_+$ such that there exists a nonempty compact subset $C$ of $\mathbb R$ satisfying $\varphi(C)\subset ]0,\infty[$. Under Assumption \ref{ergodicity}, there exists a deterministic constant $l(\varphi) > 0$ such that
\begin{displaymath}
\frac{1}{T}
\int_{0}^{T}
\varphi(X(t))dt
\xrightarrow[T\rightarrow\infty]{\normalfont{\textrm{a.s./L$^2$}}}
l(\varphi) > 0.
\end{displaymath}
\end{proposition}
%


%
\section{Convergence of the Nadaraya-Watson estimator of the drift function}
This section deals with the consistency and rate of convergence of the Nadaraya-Watson estimator of the drift function $b$ in Equation (\ref{main_equation}).
\\
\\
In the sequel, the kernel $K$ fulfills the following assumption.
%


%
\begin{assumption}\label{assumption_kernel}
$\normalfont{\textrm{supp}}(K) = [-1,1]$ and $K\in C_{b}^{1}(\mathbb R,\mathbb R_+)$.
\end{assumption}

\subsection{Why is pathwise integral inadequate}
First of all, let us prove that, even if it seems very natural, the pathwise Nadaraya-Watson estimator
\begin{displaymath}
\widetilde b_{T,h}(x) :=
\frac{\displaystyle{\int_{0}^{T}K\left(\frac{X(s) - x}{h}\right)dX(s)}}
{\displaystyle{\int_{0}^{T}K\left(\frac{X(s) - x}{h}\right)ds}} = \frac{\displaystyle{\frac 1{Th}\int_{0}^{T}K\left(\frac{X(s) - x}{h}\right)dX(s)}}{\widehat f_{T,h}(x)}
\end{displaymath}
where 
\begin{equation}\label{density}
\widehat f_{T,h}(x) :=
\frac{1}{Th}
\int_{0}^{T}
K\left(\frac{X(s) - x}{h}\right)ds.
\end{equation}
is not consistent.  
\\
\\
For this, we need the following lemma providing a convergence result for $\widehat f_{T,h}(x)$. It will also be used to prove Proposition \ref{rate_of_convergence_NW}.
%


%
\begin{lemma}\label{convergence_density}
Under Assumptions \ref{ergodicity} and \ref{assumption_kernel}, there exists a deterministic constant $l_h(x) > 0$ such that
\begin{displaymath}
\widehat f_{T,h}(x)
\xrightarrow[T\rightarrow\infty]{\normalfont{\textrm{a.s./L$^2$}}}
l_h(x) > 0.
\end{displaymath}
\end{lemma}
%


%
\begin{proof}
Under Assumption \ref{assumption_kernel}, the map
\begin{displaymath}
y\in\mathbb R\longmapsto
\frac{1}{h}K\left(\frac{y - x}{h}\right)
\end{displaymath}
satisfies the condition on $\varphi$ of Proposition \ref{ergodic_theorem}, which applies thus here and gives the result. 
\end{proof}
\noindent
Now, we state the result proving that $\widetilde b_{T,h}(x)$ is not consistent to recover $b(x)$.
%


%
\begin{proposition}\label{pathwise_NW}
Under Assumptions \ref{ergodicity} and \ref{assumption_kernel}:
\begin{displaymath}
\widetilde b_{T,h}(x)
\xrightarrow[T\rightarrow\infty]{\mathbb P} 0.
\end{displaymath}
\end{proposition}
%


%
\begin{proof}
Let $\mathcal K$ be a primitive function of $K$. By the change of variable formula for Young's integral (Proposition \ref{change_of_variable}):
\begin{eqnarray*}
 \mathcal K\left(\frac{X(T) - x}{h}\right) -
 \mathcal K\left(\frac{X(0) - x}{h}\right) & = &
 \frac{1}{h}
 \int_{0}^{T}K\left(\frac{X(s) - x}{h}\right)dX(s)\\
 & = &
 T\widehat f_{T,h}(x)\widetilde b_{T,h}(x).
\end{eqnarray*}
Then,
\begin{displaymath}
\widetilde b_{T,h}(x) =
\frac{1}{T\widehat f_{T,h}(x)}\left(\mathcal K\left(\frac{X(T) - x}{h}\right) -
\mathcal K\left(\frac{X(0) - x}{h}\right)\right).
\end{displaymath}
Since $\mathcal K$ is differentiable with bounded derivative $K$:
\begin{displaymath}
|\widetilde b_{T,h}(x)|
\leqslant
\frac{\|K\|_{\infty}}{Th\widehat f_{T,h}(x)}
|X(T) - X(0)|.
\end{displaymath}
Finally, as we know by Hairer \cite{HAIRER05}, Proposition 3.12 that
\begin{displaymath}
t\in\mathbb R_+\longmapsto\mathbb E(|X(t)|)
\end{displaymath}
is uniformly bounded, and by Lemma \ref{convergence_density} that $\widehat f_{T,h}(x)$ converges almost surely to $l_h(x) > 0$ as $T\rightarrow\infty$, it follows that $\widetilde b_{T,h}(x)$ converges to $0$ in probability, when $T\rightarrow\infty$.
\end{proof}
\noindent
This is why the Skorokhod integral replaces the pathwise stochastic integral in $\widehat b_{T,h}(x)$.

\subsection{Convergence of the Nadaraya-Watson estimator}

This subsection deals with the consistency and  rate of convergence of the estimators.
\\
\\
The Nadaraya-Watson estimator $\widehat b_{T,h}(x)$ defined by Equation (\ref{main_estimator}) can be decomposed as follows:
\begin{equation}\label{estimator_decomposition}
\widehat b_{T,h}(x) - b(x) =
\frac{B_{T,h}(x)}{\widehat f_{T,h}(x)} +
\frac{S_{T,h}(x)}{\widehat f_{T,h}(x)},
\end{equation}
where $\widehat f_{T,h}(x)$ is defined by (\ref{density}), 
\begin{displaymath}
B_{T,h}(x) :=
\frac{1}{Th}
\int_{0}^{T}
K\left(\frac{X(s) - x}{h}\right)(b(X(s)) - b(x))ds.
\end{displaymath}
and
\begin{displaymath}
S_{T,h}(x) :=
\frac{\sigma}{Th}
\int_{0}^{T}
K\left(\frac{X(s) - x}{h}\right)\delta B(s).
\end{displaymath}
By using the Lipschitz assumption \ref{ergodicity} on $b$  together with the technical lemmas proved in Section 2, the  estimators
$
\widehat b_{T,h}(x)\textrm{ and }
\widehat b_{T,h,\varepsilon}(x)
$
can be studied.
\\
\\

\begin{proposition}\label{rate_of_convergence_NW}
Under Assumptions \ref{ergodicity} and \ref{assumption_kernel},
\begin{displaymath}
|\widehat b_{T,h}(x) - b(x)|
\leqslant
\|b\|_{\normalfont{\textrm{Lip}}}h + \frac{|S_{T,h}(x)|}{\widehat f_{T,h}(x)},
\end{displaymath}
and there exists a positive constant $C$ such that
\begin{displaymath}
{\mathbb E}(S_{T,h}(x)^2)
\leqslant\frac{C}{h^4 T^{2(1-H)}}.
\end{displaymath}
As a consequence, for fixed $h>0$, we have 
\begin{equation}\label{VTh}
T^{\beta}V_{T,h}(x)
\xrightarrow[T\rightarrow\infty]{\mathbb P}0
\textrm{ $;$ }
\forall\beta\in [0,1 - H[, \mbox{ where } \; 
V_{T,h}(x) :=
\left|\frac{S_{T,h}(x)}{\widehat f_{T,h}(x)}\right|.
\end{equation}
Moreover, for $\widehat b_{T,h, \varepsilon}$ defined by (\ref{b_epsilon}),  $\forall \varepsilon>0$, 
\begin{equation}\label{avecepsilon}
|\widehat b_{T,h,\varepsilon}(x)-\widehat b_{T,h}(x)|\leqslant C \frac{\varepsilon h^{-2}T^{2H-2}}{\widehat f_{T,h}(x)}.
\end{equation}
\end{proposition}
\noindent
Heuristically, Proposition \ref{rate_of_convergence_NW} says that the pointwise quadratic risk of the kernel estimator $\widehat b_{T,h}(x)$ involves a squared bias of order $h^2$ and a variance term of order $1/(h^4T^{2(1 - H)})$. The best possible rate is thus $T^{-\frac 23 (1-H)}$ with a bandwidth choice of order 
$T^{-\frac 13 (1-H)}$. A more rigorous formulation of this is stated below.\\
Note also that it follows from (\ref{avecepsilon}) that the rate of $\widehat b_{T,h,\varepsilon}(x)$ is preserved for any small $\varepsilon$.\\
We want to emphasize that no order condition is set on the kernel, and the bias term is not bounded in the usual way for kernel setting (see e.g. Tsybakov \cite{Tsyb}, Chapter 1). Indeed, we can not refer to the expectation of the numerator as a convolution product, because the existence of a stationary density is not ensured. Would it exist, it would be difficult to set adequate regularity conditions on it. 
\\ 

\noindent
Now, consider a decreasing function $h : [t_0,\infty[\rightarrow ]0,1[$ ($t_0\in\mathbb R_+$) such that
\begin{displaymath}
\lim_{T\rightarrow\infty}
h(T) = 0
\textrm{ and }
\lim_{T\rightarrow\infty}
Th(T) =\infty
\end{displaymath}
and assume that $\widehat f_{T,h(T)}(x)$ fulfills the following assumption.
%


%
\begin{assumption}\label{convergence_density_hT}
There exists $l(x)\in ]0,\infty]$ such that $\widehat f_{T,h(T)}(x)$ converges to $l(x)$ in probability as $T\rightarrow\infty$.
\end{assumption}
\noindent
Subsection \ref{fracGauss} deals with the special case of fractional SDE with Gaussian solution in order to prove that Assumption \ref{convergence_density_hT} holds in this setting.
\\
In Proposition \ref{convergence_NW_time_dependent_h}, the result of Proposition \ref{rate_of_convergence_NW} is extended to the estimator $\widehat b_{T,h(T)}(x)$ under Assumption \ref{convergence_density_hT}.
%


%
\begin{proposition}\label{convergence_NW_time_dependent_h}
Under Assumptions \ref{ergodicity}, \ref{assumption_kernel} and \ref{convergence_density_hT}:
\begin{enumerate}
 \item If there exists $\beta\in ]0,1 - H[$ such that $T^{-\beta} =_{T\rightarrow\infty} o(h(T)^2)$, then
 \begin{displaymath}
 \widehat b_{T,h(T)}(x)
 \xrightarrow[T\rightarrow\infty]{\mathbb P} b(x).
 \end{displaymath}
 \item For every $\gamma\in ]0,\beta[$ such that
 \begin{displaymath}
 h(T) =_{T\rightarrow\infty} o(T^{-\gamma})
 \textrm{ and }
 T^{H - 1 +\gamma} =_{T\rightarrow\infty} o(h(T)^2),
 \end{displaymath}
 then
 \begin{displaymath}
 T^{\gamma}|\widehat b_{T,h(T)}(x) - b(x)|
 \xrightarrow[T\rightarrow\infty]{\mathbb P} 0.
 \end{displaymath}
\end{enumerate}
\end{proposition}
\noindent
\textbf{Example.} Consider
\begin{displaymath}
\beta\in
\left]\frac{2}{3}(1 - H),1 - H\right[
\textrm{ and }
h(T) := T^{\frac{H - 1}{3}}.
\end{displaymath}
\begin{itemize}
 \item $Th(T) = T^{H +\frac{2}{3}(1 - H)}\xrightarrow[T\rightarrow\infty]{}\infty$.
 \item $T^{-\beta}/h(T)^2 = T^{-\beta +\frac{2}{3}(1 - H)}\xrightarrow[T\rightarrow\infty]{} 0$.
 \item For every $\gamma\in ]0,(1 - H)/3[$, $h(T)/T^{-\gamma} = T^{\frac{H - 1}{3} +\gamma}\xrightarrow[T\rightarrow\infty]{} 0$.
 \item For every $\gamma\in ]0,(1 - H)/3[$, $T^{H - 1 +\gamma}/h(T)^2 = T^{\frac{H - 1}{3} +\gamma}\xrightarrow[T\rightarrow\infty]{} 0$.
\end{itemize}
\noindent
In Corollary \ref{convergence_NW_time_dependent_h_epsilon}, the result of Proposition \ref{convergence_NW_time_dependent_h} is extended to $\widehat b_{T,h(T),\varepsilon(T)}(x)$ where
\begin{displaymath}
\lim_{T\rightarrow\infty}
\varepsilon(T) = 0.
\end{displaymath}
%


%
\begin{corollary}\label{convergence_NW_time_dependent_h_epsilon}
Under Assumptions \ref{ergodicity}, \ref{assumption_kernel} and \ref{convergence_density_hT}:
\begin{enumerate}
 \item If there exists $\beta\in ]0,1 - H[$ such that
 \begin{displaymath}
 T^{-\beta} =_{T\rightarrow\infty} o(h(T)^2)
 \textrm{ and }
 \varepsilon(T) =_{T\rightarrow\infty} o(h(T)^{-2}T^{2H - 2}),
 \end{displaymath}
 then
 \begin{displaymath}
 \widehat b_{T,h(T),\varepsilon(T)}(x)
 \xrightarrow[T\rightarrow\infty]{\mathbb P} b(x).
 \end{displaymath}
 \item For every $\gamma\in ]0,\beta[$ such that
 \begin{displaymath}
 \left\{
 \begin{array}{rcl}
  h(T) & =_{T\rightarrow\infty} & o(T^{-\gamma}),\\
  T^{H - 1 +\gamma} & =_{T\rightarrow\infty} & o(h(T)^2)\\
  \varepsilon(T) & =_{T\rightarrow\infty} & o(h(T)^{-2}T^{2H - 2 +\gamma})
 \end{array}\right.,
 \end{displaymath}
 then
 \begin{displaymath}
 T^{\gamma}|\widehat b_{T,h(T),\varepsilon(T)}(x) - b(x)|
 \xrightarrow[T\rightarrow\infty]{\mathbb P} 0.
 \end{displaymath}
\end{enumerate}
\end{corollary}
\noindent
\textbf{Example.}  One can take $\varepsilon(T) := h(T)^2$.
%


%
\subsection{Special case of fractional SDE with Gaussian solution}\label{fracGauss}
The purpose of this subsection is to show that Assumption \ref{convergence_density_hT} holds when the drift function in Equation (\ref{main_equation}) is linear with a negative slope. Note also that if $H = 1/2$, then $\widehat f_{T,h(T)}$ is a consistent estimator of the stationary density for Equation (\ref{main_equation}) (see Kutoyants \cite{KUTOYANTS04}, Section 4.2).
\\
\\
Assume that Equation (\ref{main_equation}) has a centered Gaussian stationary solution $X$ and consider the normalized process $Y := X/\sigma_0$ where $\sigma_0 :=\sqrt{\textrm{var}(X_0)}$.
\\
\\
Throughout this subsection, $\nu$ is the standard normal density and the autocorrelation function $\rho$ of $Y$ fulfills the following assumption.
%


%
\begin{assumption}\label{conditions_autocorrelation}
$\displaystyle{\int_{0}^{T}\int_{0}^{T}
|\rho(v - u)|dvdu =_{T\rightarrow\infty}
O(T^{2H})}$.
\end{assumption}
\noindent
The following proposition ensures that under Assumption \ref{conditions_autocorrelation}, $\widehat f_{T,h(T)}$ fulfills Assumption \ref{convergence_density_hT} for every $x\in\mathbb R^*$.
%


%
\begin{proposition}\label{SC_consistency_density_estimator}
Under Assumptions \ref{ergodicity} and \ref{assumption_kernel}, if Equation (\ref{main_equation}) has a centered, Gaussian, stationary solution $X$, the autocorrelation function $\rho$ of $Y := X/\sigma_0$ satisfies Assumption \ref{conditions_autocorrelation} and $T^{2H - 2} =_{T\rightarrow\infty} o(h(T))$, then
\begin{equation}\label{SC_consistency_density_estimator_1}
\widehat f_{T,h(T)}(x)
\xrightarrow[T\rightarrow\infty]{\mathbb P}
\frac{1}{\sigma_0}\nu\left(\frac{x}{\sigma_0}\right) > 0
\end{equation}
for every $x\in\mathbb R^*$.
\end{proposition}
\noindent
Now, consider the fractional Langevin equation
\begin{equation}\label{Langevin_equation}
X(t) = X_0 -\lambda\int_{0}^{t}X(s)ds +\sigma B(t),
\end{equation}
where $\lambda,\sigma > 0$. Equation (\ref{Langevin_equation}) has a unique solution called Ornstein-Uhlenbeck's process.
\\
\\
On the one hand, the drift function of Equation (\ref{Langevin_equation}) fulfills Assumption \ref{ergodicity}. So, under Assumption \ref{assumption_kernel}, by Proposition \ref{rate_of_convergence_NW},
\begin{displaymath}
|\widehat b_{T,h}(x) +\lambda x|
\leqslant
\|b\|_{\normalfont{\textrm{Lip}}}h + V_{T,h}(x),
\end{displaymath}
where
\begin{displaymath}
T^{\beta}V_{T,h}(x)
\xrightarrow[T\rightarrow\infty]{\mathbb P}0
\textrm{ $;$ }
\forall\beta\in [0,2H - 1[.
\end{displaymath}
On the other hand, by Cheridito et al. \cite{CKM03}, Section 2, Equation (\ref{Langevin_equation}) has a centered, Gaussian, stationary solution $X$ such that:
\begin{displaymath}
X(t) =
\sigma\int_{-\infty}^{t}e^{-\lambda(t - u)}dB(u)
\textrm{ $;$ }\forall t\in\mathbb R_+.
\end{displaymath}
Moreover, by Cheridito et al. \cite{CKM03}, Theorem 2.3, the autocorrelation function $\rho$ of $Y := X/\sigma_0$ satisfies
\begin{displaymath}
\rho(T) =_{T\rightarrow\infty} O(T^{2H - 2}).
\end{displaymath}
So, $\rho$ fulfills Assumption \ref{conditions_autocorrelation}.
\\
\\
Consider $\beta\in ]0,2H - 1[$ and $\gamma\in ]0,2H - 1 -\beta[$ such that
\begin{displaymath}
h(T) =_{T\rightarrow\infty} o(T^{-\gamma})
\textrm{ and }
T^{H - 1 +\gamma} =_{T\rightarrow\infty} o(h(T)^2).
\end{displaymath}
Then,
\begin{displaymath}
\lim_{T\rightarrow\infty}
\frac{T^{2H - 2}}{h(T)} =
\lim_{T\rightarrow\infty}
h(T)T^{H - 1 -\gamma}\frac{T^{H - 1 +\gamma}}{h(T)^2} = 0.
\end{displaymath}
Therefore, by Proposition \ref{convergence_NW_time_dependent_h} and Proposition \ref{SC_consistency_density_estimator}:
\begin{displaymath}
T^{\gamma}|\widehat b_{T,h(T)}(x) +\lambda x|
\xrightarrow[T\rightarrow\infty]{\mathbb P} 0
\textrm{ $;$ }
\forall x\in\mathbb R^*.
\end{displaymath}
%


%
\section{Proofs}
%


%
\subsection{Proof of Proposition \ref{expression_Skorokhod}}
On the existence, uniqueness and regularity of the paths of the solution of Equation (\ref{main_equation}), see Lejay \cite{LEJAY10}.
\\
\\
Now, let us prove (\ref{decomposition_Skorokhod}).
\\
\\
Let $X_x$ be the solution of Equation (\ref{expression_Skorokhod}) with initial condition $x\in\mathbb R$. Consider also $\varphi\in\textrm{Lip}_{b}^{1}(\mathbb R)$ and $t > 0$. By Nualart \cite{NUALART06}, Proposition 5.2.3:
\begin{eqnarray*}
 \int_{0}^{t}\varphi(X_x(u))\delta X_x(u) & = &
 \int_{0}^{t}\varphi(X_x(u))b(X_x(u))du +\sigma\int_{0}^{t}\varphi(X_x(u))\delta B(u)\\
 & = &
 \int_{0}^{t}\varphi(X_x(u))dX_x(u)\\
 & &
 -\alpha_H\sigma
 \int_{0}^{t}
 \int_{0}^{t}\varphi'(X_x(u))\mathbf D_vX_x(u)|u - v|^{2H - 2}dvdu.
\end{eqnarray*}
Consider $u,v\in [0,t]$. On the one hand,
\begin{displaymath}
\mathbf D_vX_x(u) = 
\sigma\mathbf 1_{[0,u]}(v) +
\int_{0}^{u}b'(X_x(r))\mathbf D_vX_x(r)dr.
\end{displaymath}
Then,
\begin{displaymath}
\mathbf D_vX_x(u) =
\sigma\mathbf 1_{[0,u]}(v)\exp\left(
\int_{v}^{u}b'(X_x(r))dr\right).
\end{displaymath}
On the other hand,
\begin{displaymath}
\partial_xX_x(u) = 
1 +
\int_{0}^{u}b'(X_x(r))\partial_x X_x(r)dr.
\end{displaymath}
Then,
\begin{displaymath}
\partial_xX_x(u) =
\exp\left(
\int_{0}^{u}b'(X_x(r))dr\right).
\end{displaymath}
Therefore,
\begin{displaymath}
\mathbf D_vX_x(u) =
\sigma\mathbf 1_{[0,u]}(v)
\frac{\partial_x X_x(u)}{\partial_x X_x(v)}
\end{displaymath}
and
\begin{eqnarray*}
 \int_{0}^{t}\varphi(X_x(u))\delta X_x(u) & = &
 \int_{0}^{t}\varphi(X_x(u))dX_x(u)\\
 & &
 -\alpha_H\sigma^2
 \int_{0}^{t}
 \int_{0}^{u}\varphi'(X_x(u))\frac{\partial_x X_x(u)}{\partial_x X_x(v)}|u - v|^{2H - 2}dvdu.
\end{eqnarray*}
%


%
\subsection{Proof of Corollary \ref{approximation_Skorokhod}}
Consider $x\in\mathbb R$ and $\varepsilon,t > 0$. For every $s\in [0,t]$,
\begin{displaymath}
\partial_x X_x(s) = 1 +
\int_{0}^{s}b'(X_x(r))\partial_x X_x(r)dr
\end{displaymath}
and, by Taylor's formula,
\begin{small}
\begin{displaymath}
 X_{x +\varepsilon}(s) - X_x(s) =\varepsilon
 +\int_{0}^{s}
 (X_{x +\varepsilon}(r) - X_x(r))
 \int_{0}^{1}b'(X_x(r) +\theta (X_{x +\varepsilon}(r) - X_x(r)))d\theta dr.
\end{displaymath}
\end{small}
\newline
So, for every $(u,v)\in [0,t]^2$ such that $v < u$,
\begin{displaymath}
\frac{\partial_x X_x(u)}{\partial_x X_x(v)} =
\exp\left(\int_{v}^{u}b'(X_x(r))dr\right)
\end{displaymath}
and
\begin{displaymath}
\frac{X_{x +\varepsilon}(u) - X_x(u)}{X_{x +\varepsilon}(v) - X_x(v)} =
\exp\left(\int_{v}^{u}\int_{0}^{1}b'(X_x(r) +\theta(X_{x +\varepsilon}(r) - X_x(r)))d\theta dr\right).
\end{displaymath}
For a given $\varphi\in\textrm{Lip}_{b}^{1}(\mathbb R)$, by Proposition \ref{expression_Skorokhod},
\begin{displaymath}
 \Delta_{\varphi}^{S}(x,\varepsilon,t)\leqslant
 \alpha_H\sigma^2\int_{0}^{t}\int_{0}^{u}
 |\varphi'(X_x(u))|\Delta_{\varphi}(x,\varepsilon,u,v)(u - v)^{2H - 2}dvdu,
\end{displaymath}
where
\begin{displaymath}
\Delta_{\varphi}^{S}(x,\varepsilon,t) :=
\left|\int_{0}^{t}\varphi(X_x(u))\delta X_x(u) - S_{\varphi}(x,\varepsilon,t)\right|
\end{displaymath}
and, for every $(u,v)\in [0,t]^2$ such that $v < u$,
\begin{displaymath}
\Delta_{\varphi}(x,\varepsilon,u,v) :=
\left|
\frac{\partial_x X_x(u)}{\partial_x X_x(v)} -
\frac{X_{x +\varepsilon}(u) - X_x(u)}{X_{x +\varepsilon}(v) - X_x(v)}\right|.
\end{displaymath}
Since $b'(\mathbb R)\subset ]-\infty,0]$ and $b$ is two times continuously differentiable,
\begin{small}
\begin{eqnarray*}
 \Delta_{\varphi}(x,\varepsilon,u,v) & = &
 \left|\exp\left(\int_{v}^{u}b'(X_x(r))dr\right)\right.\\
 & &
 -\left.
 \exp\left(\int_{v}^{u}\int_{0}^{1}b'(X_x(r) +\theta(X_{x +\varepsilon}(r) - X_x(r)))d\theta dr\right)\right|\\
 & \leqslant &
 \sup_{z\in b'(\mathbb R)}
 e^z\\
 & &
 \times
 \int_{v}^{u}\left|b'(X_x(r)) -
 \int_{0}^{1}b'(X_x(r) +\theta(X_{x +\varepsilon}(r) - X_x(r)))d\theta\right|dr\\
 & \leqslant &
 \int_{v}^{u}\int_{0}^{1}|b'(X_x(r)) -
 b'(X_x(r) +\theta(X_{x +\varepsilon}(r) - X_x(r)))|d\theta dr\\
 & \leqslant &
 \frac{\|b''\|_{\infty}}{2}\int_{v}^{u}|X_{x +\varepsilon}(r) - X_x(r)|dr.
\end{eqnarray*}
\end{small}
\newline
Consider $s\in\mathbb R_+$. By Equation (\ref{main_equation}):
\begin{eqnarray*}
 (X_{x +\varepsilon}(s) - X_x(s))^2 & = &
 \varepsilon^2 + 2\int_{0}^{s}
 (X_{x +\varepsilon}(r) - X_x(r))d(X_{x +\varepsilon} - X_x)(r)\\
 & = &
 \varepsilon^2 + 2\int_{0}^{s}
 (X_{x +\varepsilon}(r) - X_x(r))(b(X_{x +\varepsilon}(r)) - b(X_x(r)))dr.
\end{eqnarray*}
By the mean-value theorem, there exists $x_s\in\mathbb R$ such that
\begin{eqnarray*}
 \partial_s(X_{x +\varepsilon}(s) - X_x(s))^2 & = &
 2(X_{x +\varepsilon}(s) - X_x(s))^2
 \frac{b(X_{x +\varepsilon}(s)) - b(X_x(s))}{X_{x +\varepsilon}(s) - X_x(s)}\\
 & = &
 2(X_{x +\varepsilon}(s) - X_x(s))^2
 b'(x_s)\leqslant
 -2M(X_{x +\varepsilon}(s) - X_x(s))^2
\end{eqnarray*}
and then,
\begin{displaymath}
|X_{x +\varepsilon}(s) - X_x(s)|
\leqslant
\varepsilon e^{-Ms}.
\end{displaymath}
 Therefore,
\begin{eqnarray*}
 \Delta_{\varphi}(x,\varepsilon,u,v)
 & \leqslant &
 \frac{\|b''\|_{\infty}}{2}\varepsilon
 \int_{v}^{u}e^{-Mr}dr\\
 & = &
 \frac{\|b''\|_{\infty}}{2M}\varepsilon
 (e^{-Mv} - e^{-Mu})
 \leqslant
 \frac{\|b''\|_{\infty}}{2M}\varepsilon e^{-Mv}.
\end{eqnarray*}
Finally, using the above bounds, and in a second stage, the integration by parts formula, we get:
\begin{eqnarray*}
 \Delta_{\varphi}^{S}(x,\varepsilon,t) & \leqslant &
 \alpha_H\sigma^2\frac{\|b''\|_{\infty}}{2M}\varepsilon
 \int_{0}^{t}\int_{0}^{u}|\varphi'(X_x(u))|e^{-Mv}(u - v)^{2H - 2}dvdu\\
 & \leqslant &
 \alpha_H\sigma^2\frac{\|b''\|_{\infty}\|\varphi'\|_{\infty}}{2M}\varepsilon
 \int_{0}^{t}e^{-Mv}\int_{v}^{t}(u - v)^{2H - 2}dudv\\
 & = &
 \alpha_H\sigma^2\frac{\|b''\|_{\infty}\|\varphi'\|_{\infty}}{2M(2H - 1)}\varepsilon
 \int_{0}^{t}
 e^{-Mv}(t - v)^{2H - 1}dv\\
 & = &
 \alpha_H\sigma^2\frac{\|b''\|_{\infty}\|\varphi'\|_{\infty}}{2M^2}\varepsilon
 \left(\frac{t^{2H - 1}}{2H - 1} -\int_{0}^{t}e^{-Mv}(t - v)^{2H - 2}dv\right)\\
 & \leqslant &
 \alpha_H\sigma^2\frac{\|b''\|_{\infty}\|\varphi'\|_{\infty}}{2M^2(2H - 1)}\varepsilon
 t^{2H - 1} = C_{\varphi}\varepsilon t^{2H - 1}.
\end{eqnarray*}
%


%
\subsection{Proof of Proposition \ref{ergodic_theorem}}
Consider $\gamma\in ]1/2,H[$, $\delta\in ]H -\gamma,1 -\gamma[$ and $\Omega :=\Omega_-\times\Omega_+$, where $\Omega_-$ (resp. $\Omega_+$) is the completion of $C_{0}^{\infty}(\mathbb R_-,\mathbb R)$ (resp. $C_{0}^{\infty}(\mathbb R_+,\mathbb R)$) with respect to the norm $\|.\|_-$ (resp. $\|.\|_+$) defined by
\begin{displaymath}
\|\omega_-\|_- :=
\sup_{s < t\leqslant 0}
\frac{|\omega_-(t) -\omega_-(s)|}{|t - s|^{\gamma}(1 + |s| + |t|)^{\delta}}
\textrm{ $;$ }
\forall\omega_-\in\Omega_-
\end{displaymath}
(resp.
\begin{displaymath}
\|\omega_+\|_+ :=
\sup_{0\leqslant s < t}
\frac{|\omega_+(t) -\omega_+(s)|}{|t - s|^{\gamma}(1 + |s| + |t|)^{\delta}}
\textrm{ $;$ }\forall\omega_+\in\Omega_+).
\end{displaymath}
By Hairer \cite{HAIRER05}, Section 3 or more clearly by Hairer and Ohashi \cite{HO07}, Lemmas 4.1 and 4.2, there exist a Borel probability measure $\mathbb P$ on $\Omega$ and a transition kernel $P$ from $\Omega_-$ to $\Omega_+$ such that:
\begin{itemize}
 \item The process generated by $(\Omega,\mathbb P)$ is a two-sided fractional Brownian motion $\widetilde B$.
 \item For every Borel set $U$ (resp. $V$) of $\Omega_-$ (resp. $\Omega_+$),
 \begin{displaymath}
 \mathbb P(U\times V) =
 \int_U
 P(\omega_-,V)\mathbb P_-(d\omega_-)
 \end{displaymath}
 where $\mathbb P_-$ is the probability distribution of $(\widetilde B(t))_{t\in\mathbb R_-}$.
\end{itemize}
Let $I :\mathbb R\times\Omega_+\rightarrow C^0(\mathbb R_+,\mathbb R)$ be the It\^o (solution) map for Equation (\ref{main_equation}). In general, $I(x,.)$ with $x\in\mathbb R$ is not a Markov process. However, the solution of Equation (\ref{main_equation}) can be coupled with the past of the driving signal in order to bypass this difficulty. In other words, consider the enhanced It\^o map $\mathfrak I :\mathbb R\times\Omega\rightarrow C^0(\mathbb R_+,\mathbb R\times\Omega_-)$ such that for every $(x,\omega_-,\omega_+)\in\mathbb R\times\Omega$ and $t\in\mathbb R_+$,
\begin{displaymath}
\mathfrak I(x,\omega_-,\omega_+)(t) :=
(I(x,\omega_+)(t),p_{\Omega_-}(\theta(\omega_-,\omega_+)(t)))
\end{displaymath}
where $p_{\Omega_-}$ is the projection from $\Omega$ onto $\Omega_-$,
\begin{displaymath}
\theta(\omega_-,\omega_+)(t) :=
(\omega_-\sqcup\omega_+)(t +\cdot) -
(\omega_-\sqcup\omega_+)(\cdot)
\end{displaymath}
and $\omega_-\sqcup\omega_+$ is the concatenation of $\omega_-$ and $\omega_+$. By Hairer \cite{HAIRER05}, Lemma 2.12, the process $\mathfrak I(x,.)$ is Markovian and has a Feller transition semigroup $(Q(t))_{t\in\mathbb R_+}$ such that for every $t\in\mathbb R_+$, $(x,\omega_-)\in\mathbb R\times\Omega_-$ and every Borel set $U$ (resp. $V$) of $\mathbb R$ (resp. $\Omega_-$),
\begin{displaymath}
Q(t ; (x,\omega_-),U\times V) =
\int_{V}\delta_{I(x,\omega_+)(t)}(U)P(t ; \omega_-,d\omega_+)
\end{displaymath}
where $\delta_y$ is the delta measure located at $y\in\mathbb R$ and $P(t;\omega_-,.)$ is the pushforward measure of $P(\omega_-,.)$ by $\theta(\omega_-,.)(t)$.
\\
\\
In order to prove Proposition \ref{ergodic_theorem}, let us first state the following result from Hairer \cite{HAIRER05} and Hairer and Ohashi \cite{HO07}.
%


%
\begin{theorem}\label{ergodicity_modertely_irregular}
Under Assumption \ref{ergodicity}:
\begin{enumerate}
 \item (Irreducibility) There exists $\tau\in ]0,\infty[$ such that for every $(x,\omega_-)\in\mathbb R\times\Omega_-$ and every nonempty open set $U\subset\mathbb R$,
 \begin{displaymath}
 Q(\tau ; (x,\omega_-),U\times\Omega_-) > 0.
 \end{displaymath}
 \item There exists a unique probability measure $\mu$ on $\mathbb R\times\Omega_-$ such that $\mu(p_{\Omega_-}\in\cdot) =\mathbb P_-$ and
 \begin{displaymath}
 Q(t)\mu =\mu
 \textrm{ $;$ }
 \forall t\in\mathbb R_+.
 \end{displaymath}
\end{enumerate}
\end{theorem}
\noindent
For a proof of Theorem \ref{ergodicity_modertely_irregular}.(1), see Hairer and Ohashi \cite{HO07}, Proposition 5.8. For a proof of Theorem \ref{ergodicity_modertely_irregular}.(2), see Hairer \cite{HAIRER05}, Theorem 6.1 which is a consequence of Proposition 2.18, Lemma 2.20 and Proposition 3.12.
\\
\\
Since the Feller transition semigroup $Q$ has exactly one invariant measure $\mu$ by Theorem \ref{ergodicity_modertely_irregular}, $\mu$ is ergodic, and since the first component of the process generated by $Q$ is a solution of Equation (\ref{main_equation}), by the ergodic theorem for Markov processes:
\begin{eqnarray*}
 \frac{1}{T}
 \int_{0}^{T}
 \varphi(X(t))dt & = &
 \frac{1}{T}
 \int_{0}^{T}
 (\varphi\circ p_{\mathbb R})(\mathfrak I(X_0,.)(t))dt\\
 & & 
 \xrightarrow[T\rightarrow\infty]{\normalfont{\textrm{a.s./L$^2$}}}
 \mu(\varphi\circ p_{\mathbb R}).
\end{eqnarray*}
Moreover, $\mu = Q(\tau)\mu$. So,
\begin{eqnarray*}
 \mu(\varphi\circ p_{\mathbb R}) & = &
 \int_{\mathbb R\times\Omega_-}(\varphi\circ p_{\mathbb R})(x,\omega_-)(Q(\tau)\mu)(dx,d\omega_-)\\
 & = &
 \int_{\mathbb R\times\Omega_-}\varphi(x)\int_{\mathbb R\times\Omega_-}
 Q(\tau ; (\bar x,\bar\omega_-),(dx,d\omega_-))\mu(d\bar x,d\bar\omega_-)\\
 & \geqslant &
 \min_{x\in C}\varphi(x)\cdot
 \int_{C\times\Omega_-}
 \int_{C\times\Omega_-}
 Q(\tau ; (\bar x,\bar\omega_-),(dx,d\omega_-))\mu(d\bar x,d\bar\omega_-)\\
 & \geqslant &
 \min_{x\in C}\varphi(x)\cdot
 \int_{C\times\Omega_-}
 Q(\tau ; (\bar x,\bar\omega_-),\textrm{int}(C)\times\Omega_-)\mu(d\bar x,d\bar\omega_-).
\end{eqnarray*}
Since
\begin{displaymath}
Q(\tau ; (\bar x,\bar\omega_-),\textrm{int}(C)\times\Omega_-)
> 0
\textrm{ $;$ }
\forall (\bar x,\bar\omega_-)\in\mathbb R\times\Omega_-
\end{displaymath}
by Theorem \ref{ergodicity_modertely_irregular}.(1), then
\begin{displaymath}
\int_{C\times\Omega_-}
Q(\tau ; (\bar x,\bar\omega_-),\textrm{int}(C)\times\Omega_-)\mu(d\bar x,d\bar\omega_-) > 0.
\end{displaymath}
Therefore, $\mu(\varphi\circ p_{\mathbb R}) > 0$.
%


%
\subsection{Proof of Proposition \ref{rate_of_convergence_NW}}

First write that, under Assumption \ref{ergodicity}, for any $s\in [0,T]$ such that $X(s)\in [x - h,x + h]$,

\begin{displaymath}
|b(X(s)) - b(x)|\leqslant\|b\|_{\textrm{Lip}}h.
\end{displaymath}
So,
\begin{equation}\label{bias_control}
\left|\frac{B_{T,h}(x)}{\widehat f_{T,h}(x)}\right|
\leqslant
\|b\|_{\textrm{Lip}}h.
\end{equation}
\noindent
Next, the following Lemma provides a suitable control of $\mathbb E(|S_{T,h}(x)|^2)$.
%


%
\begin{lemma}\label{control_S}
Under Assumptions \ref{ergodicity} and \ref{assumption_kernel}, there exists a deterministic constant $C > 0$, not depending on $h$ and $T$, such that:
\begin{displaymath}
\mathbb E(|S_{T,h}(x)|^2)
\leqslant
CT^{2(H - 1)}h^{-4}.
\end{displaymath}
\end{lemma}
%


%
\begin{proof}
Since $K$ belongs to $C_{b}^{1}(\mathbb R,\mathbb R_+)$, the map
\begin{displaymath}
\varphi_h :
y\in\mathbb R\longmapsto
\varphi_h(y) :=
K\left(\frac{y - x}{h}\right)
\end{displaymath}
belongs to $\textrm{Lip}_{b}^{1}(\mathbb R)$. Moreover, since $K$ and $K'$ are continuous with bounded support $[-1,1]$,
\begin{displaymath}
\left(\int_{0}^{T}\mathbb E(|\varphi_h(X(s))|^{1/H})ds\right)^{2H}
\leqslant\|K\|_{\infty}^2T^{2H}
\end{displaymath}
and
\begin{displaymath}
\left(\int_{0}^{T}\mathbb E(|\varphi_h'(X(s))|^2)^{1/(2H)}ds\right)^{2H}
\leqslant\|K'\|_{\infty}^2T^{2H}h^{-2}.
\end{displaymath}
Therefore, by Theorem \ref{control_divergence_integral}, there exists a deterministic constant $C > 0$, not depending on $h$ and $T$, such that:
\begin{eqnarray*}
 \mathbb E(|S_{T,h}(x)|^2) & = &
 \frac{1}{T^2h^2}\mathbb E\left(\left|\int_{0}^{T}\varphi_h(X(s))\delta B(s)\right|^2\right)\\
 & \leqslant &
 CT^{2(H - 1)}h^{-4}.
\end{eqnarray*}
\end{proof}
\noindent
First, by Inequality (\ref{bias_control}) and Equation (\ref{estimator_decomposition}),
\begin{displaymath}
|\widehat b_{T,h}(x) - b(x)|
\leqslant
\|b\|_{\normalfont{\textrm{Lip}}}h + V_{T,h}(x)
\end{displaymath}
where $V_{T,h}(x)$ is defined by (\ref{VTh}). 
Consider $\beta\in [0,1 - H[$. By Lemma \ref{control_S}:
\begin{displaymath}
T^{2\beta}\mathbb E(|S_{T,h}(x)|^2)
\leqslant
CT^{2(H - 1 +\beta)}h^{-4}
\xrightarrow[T\rightarrow\infty]{} 0.
\end{displaymath}
So,
\begin{displaymath}
T^{\beta}|S_{T,h}(x)|
\xrightarrow[T\rightarrow\infty]{\mathbb P}0.
\end{displaymath}
Moreover, by Lemma \ref{convergence_density}:
\begin{displaymath}
\frac{1}{\widehat f_{T,h}(x)}
\xrightarrow[T\rightarrow\infty]{\mathcal D}
\frac{1}{l_h(x)} > 0.
\end{displaymath}
Therefore, by Slutsky's lemma:
\begin{displaymath}
T^{\beta}V_{T,h}(x)
\xrightarrow[T\rightarrow\infty]{\mathbb P}
0.
\end{displaymath}
Lastly, the bound (\ref{avecepsilon}) follows from the following Lemma.
\begin{lemma}\label{control_NW_epsilon}
Under Assumptions \ref{ergodicity} and \ref{assumption_kernel}, there exists a deterministic constant $C > 0$, not depending on $\varepsilon$, $h$ and $T$, such that:
\begin{displaymath}
|\widehat b_{T,h,\varepsilon}(x) -\widehat b_{T,h}(x)|
\leqslant
C\frac{\varepsilon h^{-2}T^{2H - 2}}{\widehat f_{T,h}(x)}.
\end{displaymath}
\end{lemma}
%


%
\begin{proof}
Since $K$ belongs to $C_{b}^{1}(\mathbb R,\mathbb R_+)$, the map
\begin{displaymath}
\varphi_h :
y\in\mathbb R\longmapsto
\varphi_h(y) :=
K\left(\frac{y - x}{h}\right)
\end{displaymath}
belongs to $\textrm{Lip}_{b}^{1}(\mathbb R)$. Consider
\begin{eqnarray*}
 S_h(x_0,\varepsilon,T) & := &
 \int_{0}^{T}\varphi_h(X_{x_0}(u))dX_{x_0}(u)\\
 & &
 -\alpha_H\sigma^2
 \int_{0}^{T}
 \int_{0}^{u}\varphi_h'(X_{x_0}(u))\frac{X_{x_0 +\varepsilon}(u) - X_{x_0}(u)}{X_{x_0 +\varepsilon}(v) - X_{x_0}(v)}|u - v|^{2H - 2}dvdu.
\end{eqnarray*}
By Corollary \ref{approximation_Skorokhod}:
\begin{eqnarray*}
 \left|\int_{0}^{T}\varphi_h(X_{x_0}(u))\delta X_{x_0}(u) -
 S_h(x_0,\varepsilon,T)\right|
 & \leqslant &
 H\sigma^2\frac{\|b''\|_{\infty}\|\varphi_h'\|_{\infty}}{M^2}
 \varepsilon T^{2H - 1}\\
 & \leqslant &
 C\varepsilon h^{-1}T^{2H - 1},
\end{eqnarray*}
where
\begin{displaymath}
C :=
\frac{H\sigma^2\|b''\|_{\infty}\|K'\|_{\infty}}{M^2}.
\end{displaymath}
Therefore,
\begin{displaymath}
|\widehat b_{T,h,\varepsilon}(x) -\widehat b_{T,h}(x)|
\leqslant
C\frac{\varepsilon h^{-2}T^{2H - 2}}{\widehat f_{T,h}(x)}.
\end{displaymath}
\end{proof}


%
\subsection{Proof of Proposition \ref{convergence_NW_time_dependent_h}}
On the one hand, assume that there exists $\beta\in ]0,1 - H[$ such that
\begin{displaymath}
T^{-\beta} =_{T\rightarrow\infty} o(h(T)^2)
\end{displaymath}
in order to show the consistency of the estimator $\widehat b_{T,h(T)}(x)$. First, let us prove that
\begin{equation}\label{convergence_NW_time_dependent_h_1}
\frac{S_{T,h(T)}(x)}{\widehat f_{T,h(T)}(x)}
\xrightarrow[T\rightarrow\infty]{\mathbb P}
0.
\end{equation}
For $\varepsilon > 0$ arbitrarily chosen:
\begin{displaymath}
\mathbb P\left(
\left|\frac{S_{T,h(T)}(x)}
{\widehat f_{T,h(T)}(x)}\right|\geqslant\varepsilon\right)
\leqslant
\mathbb P(|S_{T,h(T)}(x)|\geqslant\varepsilon T^{H +\beta - 1}) +
\mathbb P(\widehat f_{T,h(T)}(x) < T^{H +\beta - 1}).
\end{displaymath}
By Lemma \ref{control_S}:
\begin{displaymath}
\mathbb P(|S_{T,h(T)}(x)|\geqslant\varepsilon T^{H +\beta - 1})
\leqslant
C\varepsilon^{-2}|h(T)^{-2}T^{-\beta}|^2
\xrightarrow[T\rightarrow\infty]{} 0.
\end{displaymath}
So, since
\begin{displaymath}
\widehat f_{T,h(T)}(x)
\xrightarrow[T\rightarrow\infty]{\mathbb P}
l(x)\in ]0,\infty],
\end{displaymath}
the convergence result (\ref{convergence_NW_time_dependent_h_1}) is true.
\\
\\
Moreover, by Inequality (\ref{bias_control}):
\begin{equation}\label{convergence_NW_time_dependent_h_2}
\frac{B_{T,h(T)}(x)}{\widehat f_{T,h(T)}(x)}
\xrightarrow[T\rightarrow\infty]{\textrm{a.s.}}
0.
\end{equation}
Therefore, by the convergence results (\ref{convergence_NW_time_dependent_h_1}) and (\ref{convergence_NW_time_dependent_h_2}) together with Equation (\ref{estimator_decomposition}):
\begin{displaymath}
\widehat b_{T,h(T)}(x)
\xrightarrow[T\rightarrow\infty]{\mathbb P}b(x).
\end{displaymath}
On the other hand, let $\gamma\in ]0,\beta[$ be arbitrarily chosen such that
\begin{displaymath}
h(T) =_{T\rightarrow\infty} o(T^{-\gamma})
\textrm{ and }
T^{H - 1 +\gamma} =_{T\rightarrow\infty} o(h(T)^2)
\end{displaymath}
in order to show that
\begin{equation}\label{convergence_NW_time_dependent_h_3}
T^{\gamma}|\widehat b_{T,h(T)}(x) - b(x)|
\xrightarrow[T\rightarrow\infty]{\mathcal D} 0.
\end{equation}
First, by Inequality (\ref{bias_control}) and Equation (\ref{estimator_decomposition}):
\begin{equation}\label{convergence_NW_time_dependent_h_4}
T^{\gamma}|\widehat b_{T,h(T)}(x) - b(x)|
\leqslant
\|b\|_{\normalfont{\textrm{Lip}}}T^{\gamma}h(T) + T^{\gamma}V_{T,h(T)}(x).
\end{equation}
By Lemma \ref{control_S}:
\begin{displaymath}
T^{2\gamma}\mathbb E(|S_{T,h(T)}(x)|^2)
\leqslant
C|h(T)^{-2}T^{H - 1 +\gamma}|^2
\xrightarrow[T\rightarrow\infty]{} 0.
\end{displaymath}
So, since
\begin{displaymath}
\widehat f_{T,h(T)}(x)
\xrightarrow[T\rightarrow\infty]{\mathbb P}
l(x)\in ]0,\infty],
\end{displaymath}
by Slutsky's lemma:
\begin{displaymath}
T^{\gamma}V_{T,h(T)}(x)
\xrightarrow[T\rightarrow\infty]{\mathcal D}
0.
\end{displaymath}
Finally, since $h(T) =_{T\rightarrow\infty} o(T^{-\gamma})$, by Equation (\ref{convergence_NW_time_dependent_h_4}), the convergence result (\ref{convergence_NW_time_dependent_h_3}) is true.
%


%
\subsection{Proof of Corollary \ref{convergence_NW_time_dependent_h_epsilon}}
In order to establish a rate of convergence for $\widehat b_{T,h,\varepsilon}(x)$, Lemma \ref{control_S} and Lemma \ref{control_NW_epsilon} provide a suitable control.
%


%
%
\noindent
Indeed, by Lemma \ref{control_NW_epsilon}, there exists a deterministic constant $C > 0$ such that:
\begin{eqnarray*}
 |\widehat b_{T,h(T),\varepsilon(T)}(x) - b(x)|
 & \leqslant &
 |\widehat b_{T,h(T),\varepsilon(T)}(x) -\widehat b_{T,h(T)}(x)| +
 |\widehat b_{T,h(T)}(x) - b(x)|\\
 & \leqslant &
  C
 \frac{\varepsilon(T)h(T)^{-2}T^{2H - 2}}{\widehat f_{T,h(T)}(x)} +
 |\widehat b_{T,h(T)}(x) - b(x)|.
\end{eqnarray*}
Proposition \ref{convergence_NW_time_dependent_h} allows to conclude.
%


%
\subsection{Proof of Proposition \ref{SC_consistency_density_estimator}}
 Consider a random variable $U\rightsquigarrow\mathcal N(0,1)$ and
\begin{displaymath}
\mathcal G :=
\{G :\mathbb R\rightarrow\mathbb R :
\mathbb E(G(U)) = 0
\textrm{ and }\mathbb E(G(U)^2) <\infty\},
\end{displaymath}
which is a subset of $L^2(\mathbb R,\nu(y)dy)$.
\\
\\
The Hermite polynomials
\begin{displaymath}
H_q(y) :=
(-1)^qe^{y^2/2}\frac{d^q}{dy^q}e^{-y^2/2}
\textrm{ $;$ }
y\in\mathbb R
\textrm{, }
q\in\mathbb N
\end{displaymath}
form a complet orthogonal system of functions of $L^2(\mathbb R,\nu(y)dy)$ such that
\begin{displaymath}
\mathbb E(H_q(U)H_p(U)) = q!\delta_{p,q}
\textrm{ $;$ }
\forall p,q\in\mathbb N.
\end{displaymath}
By Taqqu \cite{TAQQU75} (see p. 291) and Puig et al. \cite{PPS02}, Lemma 3.3:
\begin{enumerate}
 \item For any $G\in\mathcal G$ and $y\in\mathbb R$,
 \begin{equation}\label{Hermite_decomposition}
 G(y) =\sum_{q = m(G)}^{\infty}\frac{J(q)}{q!}H_q(y)
 \end{equation}
 in $L^2(\mathbb R,\nu(y)dy)$, where
 \begin{displaymath}
 J(q) :=\mathbb E(G(U)H_q(U))
 \textrm{ $;$ }\forall q\in\mathbb N
 \end{displaymath}
 and
 \begin{displaymath}
 m(G) :=
 \inf\{q\in\mathbb N : J(q)\not= 0\}.
 \end{displaymath}
 \item (Mehler's formula) For any centered, normalized and stationary Gaussian process $Z$ of autocorrelation function $R$:
 \begin{equation}\label{Mehler_formula}
 \mathbb E(H_q(Z(u))H_p(Z(v))) =
 q!R(v - u)^q\delta_{p,q}
 \textrm{ $;$ }
 \forall u,v\in\mathbb R_+
 \textrm{, }
 \forall p,q\in\mathbb N.
 \end{equation}
\end{enumerate}
Consider the map $K_T :\mathbb R\rightarrow\mathbb R$ defined by:
\begin{displaymath}
K_T(y) :=
\frac{1}{h(T)}K\left(\frac{y}{h(T)}\right)
\textrm{ $;$ }
\forall y\in\mathbb R.
\end{displaymath}
In order to use (\ref{Hermite_decomposition}) and (\ref{Mehler_formula}) to prove the convergence result (\ref{SC_consistency_density_estimator_1}), note that $\widehat f_{T,h(T)}(x)$ can be rewritten as
\begin{displaymath}
\widehat f_{T,h(T)}(x) =
\frac{1}{T}
\int_{0}^{T}G_{T,x}(Y(s))ds - R_{T,x},
\end{displaymath}
where
\begin{displaymath}
R_{T,x} :=
\frac{1}{\sigma_0}\left(K_T\ast\nu\left(\frac{.}{\sigma_0}\right)\right)(x)
\textrm{ $;$ }
\forall y\in\mathbb R
\end{displaymath}
and
\begin{displaymath}
G_{T,x}(y) :=
K_T(\sigma_0y - x) - R_{T,x}.
\end{displaymath}
%


%
\begin{lemma}\label{Hermite_rank_G}
The map $G_{T,x}$ belongs to $\mathcal G$ and there exists $T_x > 0$ such that
\begin{displaymath}
m(G_{T,x}) = 1
\textrm{ $;$ }
\forall T > T_x.
\end{displaymath}
\end{lemma}
%


%
\begin{proof}
On the one hand, since $K_T$ is continuous and its support is compact, $G_{T,x}\in L^2(\mathbb R,\nu(y)dy)$. Moreover,
\begin{eqnarray*}
 \mathbb E(G_{T,x}(U)) & = &
 \int_{-\infty}^{\infty}G_{T,x}(y)\nu(y)dy\\
 & = &
 \int_{-\infty}^{\infty}K_T(\sigma_0y - x)\nu(y)dy - R_{T,x} = 0.
\end{eqnarray*}
So, $G_{T,x}\in\mathcal G$.
\\
\\
On the other hand, for every $q\in\mathbb N$, by putting $J_{T,x}(q) :=\mathbb E(G_{T,x}(U)H_q(U))$,
\begin{eqnarray*}
 J_{T,x}(1) & = &
 \int_{-\infty}^{\infty}
 G_{T,x}(y)H_1(y)\nu(y)dy\\
 & = &
 \int_{(x - h(T))/\sigma_0}^{(x + h(T))/\sigma_0}
 K_T(\sigma_0y - x)\nu(y)ydy -
 R_{T,x}\int_{-\infty}^{\infty}H_0(y)H_1(y)\nu(y)dy\\
 & = &
 \int_{(x - h(T))/\sigma_0}^{(x + h(T))/\sigma_0}
 K_T(\sigma_0y - x)\nu(y)ydy.
\end{eqnarray*}
For any $x > 0$, there exists $T_{x}^{+} > 0$ such that for every $T > T_{x}^{+}$,
\begin{displaymath}
I_{T,x} :=
\left[\frac{x - h(T)}{\sigma_0} ; \frac{x + h(T)}{\sigma_0}\right]
\subset ]0,\infty[.
\end{displaymath}
For every $T > T_{x}^{+}$, since $y\mapsto K_T(\sigma_0y - x)$, $\nu$ and $\textrm{Id}_{\mathbb R}$ are continuous and strictly positive on $I_{T,x}^{\circ}$, $J_{T,x}(1) > 0$. Symmetrically, for every $x < 0$, there exists $T_{x}^{-} > 0$ such that for every $T > T_{x}^{-}$, $J_{T,x}(1) < 0$. This concludes the proof.
\end{proof}
%


%
\begin{lemma}\label{asymptotic_behavior_J}
For every $x\in\mathbb R^*$,
\begin{displaymath}
\sum_{q = 1}^{\infty}
\frac{J_{T,x}(q)^2}{q!} =_{T\rightarrow\infty}
O\left(\frac{1}{h(T)}\right).
\end{displaymath}
\end{lemma}
%


%
\begin{proof}
Since $G_{T,x}\in L^2(\mathbb R,\nu(y)dy)$, by Parseval's inequality:
\begin{eqnarray*}
 \sum_{q = 1}^{\infty}
 \frac{J_{T,x}(q)^2}{q!} & = &
 \mathbb E(G_{T,x}(U)^2)\\
 & = &
 \int_{-\infty}^{\infty}
 (K_T(\sigma_0y - x) - R_{T,x})^2\nu(y)dy\\
 & \leqslant &
 2\int_{-\infty}^{\infty}
 K_T(\sigma_0y - x)^2\nu(y)dy +
 2R_{T,x}^{2}.
\end{eqnarray*}
On the one hand,
\begin{displaymath}
R_{T,x}
\xrightarrow[T\rightarrow\infty]{}
\frac{1}{\sigma_0}\nu\left(\frac{x}{\sigma_0}\right).
\end{displaymath}
So,
\begin{displaymath}
R_{T,x}^{2} =_{T\rightarrow\infty} O(1).
\end{displaymath}
On the other hand,
\begin{eqnarray*}
 \int_{-\infty}^{\infty}K_T(\sigma_0y - x)^2\nu(y)dy & = &
 \frac{1}{\sigma_0h(T)}\int_{-1}^{1}K(y)^2\nu\left(\frac{h(T)y + x}{\sigma_0}\right)dy\\
 & \leqslant &
 \frac{2\|K\|_{\infty}^{2}\|\nu\|_{\infty}}{\sigma_0h(T)}.
\end{eqnarray*}
Therefore,
\begin{displaymath}
\sum_{q = 2}^{\infty}\frac{J_{T,x}(q)^2}{q!} =_{T\rightarrow\infty}
O\left(\frac{1}{h(T)}\right).
\end{displaymath}
\end{proof}
\noindent
In order to prove the convergence result (\ref{SC_consistency_density_estimator_1}), since
\begin{displaymath}
R_{T,x}
\xrightarrow[T\rightarrow\infty]{}
\frac{1}{\sigma_0}\nu\left(\frac{x}{\sigma_0}\right),
\end{displaymath}
let us prove that
\begin{equation}\label{SC_consistency_density_estimator_2}
\left|\frac{1}{T}\int_{0}^{T}G_{T,x}(Y(s))ds\right|
\xrightarrow[T\rightarrow\infty]{\textrm L^2} 0.
\end{equation}
By the decomposition (\ref{Hermite_decomposition}) and Mehler's formula (\ref{Mehler_formula}) applied to $G_{T,x}$ and $Y$, for every $u,v\in [0,T]$,
\begin{displaymath}
\mathbb E(G_{T,x}(Y(u))G_{T,x}(Y(v))) =
\sum_{q = 1}^{\infty}\frac{J_{T,x}(q)^2}{q!}\rho(v - u)^q.
\end{displaymath}
So, since $\rho$ is a $[-1,1]$-valued function,
\begin{eqnarray*}
 \mathbb E\left(
 \left|\int_{0}^{T}G_{T,x}(Y(s))ds\right|^2
 \right) & = &
 \int_{0}^{T}\int_{0}^{T}
 |\mathbb E(G_{T,x}(Y(u))G_{T,x}(Y(v)))|dudv\\
 & \leqslant &
 \sum_{q = 1}^{\infty}\frac{J_{T,x}(q)^2}{q!}
 \int_{0}^{T}\int_{0}^{T}|\rho(v - u)|^qdudv\\
 & \leqslant &
 \left(\int_{0}^{T}\int_{0}^{T}|\rho(v - u)|dudv\right)
 \sum_{q = 1}^{\infty}\frac{J_{T,x}(q)^2}{q!}.
\end{eqnarray*}
Then, by Assumption \ref{conditions_autocorrelation} and Lemma \ref{asymptotic_behavior_J}:
\begin{displaymath}
\lim_{T\rightarrow\infty}
\mathbb E\left(
\left|
\frac{1}{T}
\int_{0}^{T}G_{T,x}(Y(s))ds\right|^2\right)
=\lim_{T\rightarrow\infty}
\frac{T^{2H - 2}}{h(T)} = 0.
\end{displaymath}
Therefore, the convergence result (\ref{SC_consistency_density_estimator_2}) is true.
%


%

%
\end{document}